\documentclass[10pt]{amsart}
\usepackage{amsmath}
\usepackage{amsfonts}
\usepackage{amssymb}

\newtheorem{lemma}{Lemma}
\newtheorem{theo}{Theorem}
\newtheorem{claim}{Claim}

\newtheorem{rem}{Remark}
\newcommand{\abs}[1]{\ensuremath{\left| #1 \right|}}
\newcommand{\sett}[1]{\ensuremath{\left \{ #1 \right \}}}
\newcommand{\ip}[2]{\ensuremath{\left<#1,#2\right>}}
\newcommand{\RRd}{{\mathbb{R}^d}}
\newcommand{\Zdst}{{\mathbb{Z}^d}}
\newcommand{\Rdst}{{\mathbb{R}^d}}
\newcommand{\norm}[1]{\lVert#1\rVert}
\newcommand{\bC}{{\mathbb{C}}}

\title{Explicit localization estimates for spline-type spaces}
\author{Jos\'e Luis Romero}
\address{Departamento de
Matem\'atica \\ Facultad de Ciencias Exactas y Naturales\\ Universidad
de Buenos Aires\\ Ciudad Universitaria, Pabell\'on I\\ 1428 Capital
Federal\\ ARGENTINA\\ and CONICET, Argentina}
\email[Jos\'e Luis Romero]{jlromero@dm.uba.ar}

\begin{document}
\maketitle

\begin{abstract}
In this article we derive some explicit decay estimates for the dual system of a basis of functions polynomially localized in space.
\end{abstract}
\section{Introduction}
A spline-type space $S$ is a closed subspace of $L^2(\RRd)$ possessing a Riesz basis of functions well localized is space. That is, there exists a family of functions $\sett{f_k}_k \subseteq S$ and constants
$0<A \leq B < +\infty$ such that
\begin{equation}
\label{expan}
A \norm{c}^2_{\ell^2} \leq \norm{\sum_k c_k f_k}^2_{L^2} \leq B \norm{c}^2_{\ell^2},
\end{equation}
holds for every $c \in \ell^2$, and the functions $\sett{f_k}_k$ satisfy an spatial localization condition\footnote{In particular, the series in the equation is required to converge unconditionally.}.

In a spline-type space any function in $f \in S$ has a unique expansion $f=\sum_k c_k f_k$. Moreover, the coefficients are given by $c_k=\ip{f}{g_k}$, where $\sett{g_k}_k \subseteq S$ is the \emph{dual basis},
a set of functions characterized by the biorthogonality relation $\ip{g_k}{f_j} = \delta_{k,j}$.

The general theory of localized frames (see \cite{gr04}, \cite{fogr05} and \cite{bacahela06}) asserts that the functions forming the dual basis satisfy a similar spatial localization. This can be used to extend
the expansion in \eqref{expan} to other spaces, so that the family $\sett{f_k}_k$ becomes a Banach frame for an associated family of Banach spaces (see \cite{fegr89} and \cite{gr04}). In the case of a spline-type space $S$, this means that the decay of a function in $S$ can be characterized by the decay of its coefficients and that, in particular, the functions $\sett{f_k}_k$ form a so called $p$-Riesz basis for its $L^p$-closed linear span, for the whole range $1 \leq p \leq \infty$.

The purpose of this article is to derive, in some concrete case, explicit bounds for the localization of the dual basis. We will work with a set of functions satisfying a polynomial decay condition around a set of nodes forming a lattice. By a change of variables, we can assume that the lattice is $\Zdst$. So, we will consider a set of functions $\sett{f_k}_k \subseteq L^2(\RRd)$ satisfying the condition,
\begin{equation}
\label{nodes}
\abs{f_k(x)} \leq C \left(1+\abs{x-k} \right)^{-s},
\quad
x \in \RRd
\mbox{ and }
k \in \Zdst,
\end{equation}
for some constant $C$. This type of spatial localization is specifically covered by the results in \cite{gr04}, but the constants given there are not explicit. We will derive a polynomial decay condition for the dual basis $\sett{g_k}_k$, giving explicit information on the resulting constants. This yields some qualitative information, like the dependence of theses constants on $A,C$ and $s$ and the corresponding $p$-Riesz basis bounds for the original basis. In particular, it implies that if a family of Riesz basic sequences is given, where all the functions satisfy the concentration condition in equation \eqref{nodes} uniformly, and have a uniform lower basis bound, then the corresponding dual systems are also uniformly concentrated in space.

Since the localization condition in equation \eqref{nodes} is stable under small perturbations, the explicit estimates in Theorem \eqref{main} can be used to derive various kinds of stability conclusions. In particular, it could be used to derive jitter-error estimates for sampling in general spline-type spaces.

The results in \cite{gr04} prescribe polynomial decay estimates for the dual basis similar to those possessed by the original basis. As a trade-off for the explicit constants we will not obtain the full preservation of these decay conditions (see Remark \ref{rates}). Nevertheless, any degree of polynomial decay on the dual system can be granted, provided that the original basis has sufficiently good decay.

Finally observe that, although the basis $\sett{f_k}_k$ is assumed to be concentrated around a lattice of nodes, the functions $f_k$ are not assumed to be shifts of a single function. In particular, Theorem \ref{main} below allows for a basis of functions whose `optimal' concentration nodes do not form a lattice but are comparable to one - for example, it is possible to consider a system of functions constructed as integer shifts of a single one and then translate arbitrarily finitely many of them. The `eccentricity' of the configuration of concentration nodes is, however, penalized by the constants modelling the decay.

\section{Assumptions and statements}
\begin{theo}
\label{main}
Let $C \geq 1$, and let $t>d$ be integers. Let $s>d+t$ be a real number.
For $k \in \Zdst$, let $f_k: \Rdst \to \bC$ be a measurable function such that
\[
\abs{f_k(x)} \leq C \left(1+\abs{x-k} \right)^{-s},
\qquad (x \in \RRd).
\]
Suppose that $\sett{f_k}_k$ is a Riesz basis for $S$, its 
closed linear span within $L^2$, with bounds $0<A \leq B < \infty$
(that is, equation \eqref{expan} holds.)
Then, the dual functions satisfy,
\[
\abs{g_k(x)} \leq D \left(1+\abs{x-k} \right)^{-t},
\qquad (x \in \RRd).
\]
where D is given by,
\[
D = \frac{E^{t^2} C^{2t+1}}{A^{t+1}} \left( 1+ \frac{1}{s-t-d} \right)^t,
\]
for some constant $E>0$ that only depends on the dimension $d$.
\end{theo}
\begin{rem}
The constant $E$ can be explicitly determined from the proof.
\end{rem}
Of course, by a change of variables, Theorem \ref{main} can be extended
to a general lattice of the form $P \Zdst$, where $P$ in an invertible matrix.

Since in the statement of the theorem we are allowing a constant that may depend on $d$,
we can use any norm $\abs{\cdot}$. For convenience, we will use the max-norm
$\abs{x} = \max\sett{\abs{x_1},\ldots,\abs{x_d}}$. For the rest of the note, the functions $f_k, g_k$ and
the constants $A, B, C, s, t$ will be fixed.

\section{Definitions and notation}
For a function $f$, the norm $\norm{f}$ (with no subscript), will denote its $L^2$ norm.
Given a matrix $L \equiv (l_{k,j})_{k,j \in \Zdst}$ and $1 \leq h \leq d$, let
\[
D_h(L)_{k,j} := (k_h - j_h) l_{k,j}.
\]
We denote by $\norm{L}$ the norm of $L$ as an operator on $\ell^2(\Zdst)$.

For two functions $f$ and $g$, we write $f \lesssim g$ if the exists a constant $c>0$ depending only on the dimension $d$
such that $f \leq c g$. Furthermore, for functions $f,g$ depending implicitly on $u>0$
we say what $f \lesssim_u g$ if there is a constant $c>0$ depending only on the dimension $d$
such that $f \leq c^u g$.

Let us define the matrix $M \equiv (m_{k,j})_{k,j \in \Zdst}$ by \[ m_{k,j} := \ip{f_k}{f_j}. \]
Since $\sett{f_k}_k$ is a Riesz sequence, $M$ is invertible. Moreover,
$\norm{M^{-1}} \leq A^{-1}$ and $M^{-1}_{k,j}=\ip{g_k}{g_j}$.

We also note the following estimates. The proof uses some techniques taken
from \cite{fe79}.
\begin{lemma}
\label{obs}
\begin{align*}
&\mbox{\bf (a)} \qquad
W_u := \sum_{k \in \Zdst} (1+\abs{k})^{-u} \lesssim \left(1+\frac{1}{u-d}\right),
\quad k \in \Zdst, u > d.
\\
&\mbox{\bf (b)} \qquad
\sum_j (1+\abs{k-j})^{-u} (1+\abs{j})^{-u} \lesssim (1+\abs{k})^{-u},
\quad k \in \Zdst, u \geq d+1.
\\
&\mbox{\bf (c)} \qquad
\int_{\Rdst} (1+\abs{x-y})^{-u} (1+\abs{y})^{-u} dy \lesssim (1+\abs{x})^{-u},
\quad x \in \Rdst, u \geq d+1.
\end{align*}
\end{lemma}
\begin{rem}
For items (b) and (c), the fact that $u \geq d+1$ is used to grant that the constant does not depend on $u$.
\end{rem}
\begin{proof}
First observe that for $u>d$, $\int_{\RRd} \left(1 + \abs{x} \right)^{-u} dx \lesssim \frac{1}{u-d}$.

To prove (a) we simply estimate,
\begin{align*}
\sum_{k \in \Zdst} (1+\abs{k})^{-u} &\leq \sum_{k \in \Zdst} \int_{[0,1]^d+k} (1+\abs{k})^{-u} dx.
\end{align*}
For $x \in [0,1]^d+k$, 
$1+\abs{x} \leq 2 + \abs{k} \leq 2(1+\abs{k})$,
so $(1+\abs{k})^{-u} \leq 2^u (1+\abs{x})^{-u}$. Therefore,
\begin{align*}
\sum_{k \in \Zdst} (1+\abs{k})^{-u} &\leq \sum_{k \in \Zdst} 2^u \int_{[0,1]^d+k} (1+\abs{x})^{-u} dx
\\
&\leq 2^u \int_{\RRd} (1+\abs{x})^{-u} dx \lesssim 2^u \frac{1}{u-d}.
\end{align*}
So, for $u \leq d+1$, 
\begin{align*}
\sum_k (1+\abs{k})^{-u} \lesssim \frac{1}{u-d} \leq \left(1+\frac{1}{u-d}\right).
\end{align*}
For $u > d+1$,
\begin{align*}
\sum_k (1+\abs{k})^{-u} \leq \sum_k (1+\abs{k})^{-(d+1)} 
\approx 1 \leq \left(1+\frac{1}{u-d}\right).
\end{align*}

For (b), we split the sum,
\begin{align*}
&\sum_j (1+\abs{k-j})^{-(d+1)} (1+\abs{j})^{-(d+1)}
\\
&\qquad \leq  \sum_{\abs{j}>\abs{k}/2} (1+\abs{k-j})^{-(d+1)} (1+\abs{j})^{-(d+1)}
\\
&\qquad \qquad
+ \sum_{\abs{j} \leq \abs{k}/2} (1+\abs{k-j})^{-(d+1)} (1+\abs{j})^{-(d+1)}.
\end{align*}
For the first sum, if $\abs{j}>\abs{k}/2$, then $1+\abs{j} \geq 1+\abs{k}/2 \geq 1/2(1+\abs{k})$,
and $(1+\abs{j})^{-(d+1)} \leq 2^{(d+1)} (1+\abs{k})^{-(d+1)}$.
For the second sum, if $\abs{j}\leq \abs{k}/2$, then $\abs{k-j}>\abs{k}/2$, and as before,
$(1+\abs{k-j})^{-(d+1)} \leq 2^{(d+1)} (1+\abs{k})^{-(d+1)}$.

Therefore,
\begin{align*}
&\sum_j (1+\abs{k-j})^{-(d+1)} (1+\abs{j})^{-(d+1)}
\\
&\qquad \leq 2^{(d+1)} (1+\abs{k})^{-(d+1)} \left( \sum_{\abs{j}>\abs{k}/2} (1+\abs{k-j})^{-(d+1)}
+ \sum_{\abs{j} \leq \abs{k}/2} (1+\abs{j})^{-(d+1)} \right)
\\
&\qquad \lesssim (1+\abs{k})^{-(d+1)}.
\end{align*}
We now observe that, for any $\alpha \geq 0$ and all $k \in \Zdst$,
\begin{align*}
(1+\abs{k})^\alpha \leq (1+\abs{k-j})^\alpha (1+\abs{j})^\alpha.
\end{align*}
Consequently, with $\alpha=u-d-1$,
\begin{align*}
\sum_j (1+\abs{k-j})^{-u} (1+\abs{j})^{-u} \left(1+\abs{k}\right)^\alpha
&\leq \sum_j (1+\abs{k-j})^{-(d+1)} (1+\abs{j})^{-(d+1)}
\\
&\lesssim (1+\abs{k})^{-(d+1)}.
\end{align*}
Therefore,
\begin{align*}
\sum_j (1+\abs{k-j})^{-u} (1+\abs{j})^{-u}
\lesssim (1+\abs{k})^{-u}.
\end{align*}
This is the desired estimate. Assertion (c) follows similarly.
\end{proof}

\section{Proofs}
We will first prove some lemmas and claims.
\begin{lemma}
\label{coef}
Let $\sett{c_k}_k \subseteq \bC$ be a sequence such that
\[
\abs{c_k} \leq \alpha (1+\abs{k-j})^{-t},
\quad k \in \Zdst,
\]
holds for some constant $\alpha>0$ and $j \in \Zdst$.
Then the series defining the function,
\[
f=\sum_k c_k f_k
\]
is pointwise convergent and $f$ satisfies
\[
\abs{f(x)} \lesssim_t \alpha C (1+\abs{x-j})^{-t},
\quad x \in \Rdst
\]
\end{lemma}

\begin{proof}
Given $x \in \Rdst$, let $k_0 \in \Zdst$
be such that $x \in [0,1]^d + k_0$. Then, for any $k \in \Zdst$,
$\frac{1}{2}(1 + \abs{x-k}) \leq 1 + \abs{k_0-k} \leq 2(1 + \abs{x-k})$.
Therefore,
\[
2^{-t}(1 + \abs{x-k})^{-t} \leq (1 + \abs{k_0-k})^{-t} \leq 2^t(1 + \abs{x-k})^{-t}.
\]
Now we estimate,
\begin{align*}
\sum_k \abs{c_k}\abs{f_k(x)}
&\leq \alpha C \sum_k (1+\abs{k-j})^{-t} \left(1+\abs{x-k} \right)^{-s}
\\
&\leq \alpha C \sum_k (1+\abs{k-j})^{-t} \left(1+\abs{x-k} \right)^{-t}
\\
&\leq \alpha C 2^t \sum_k (1+\abs{k-j})^{-t} \left(1+\abs{k_0-k} \right)^{-t}
\\
&\lesssim_t \alpha C (1+\abs{k_0-j})^{-t}
\lesssim_t \alpha C (1+\abs{x-j})^{-t},
\end{align*}
where we used that $s>t$ and that $t-d \geq 1$.

This proves that the series defining $f$ is absolutely convergent at $x$ and that $f$ satisfies
the required inequality.
\end{proof}

\begin{claim}
\label{claim1}
Let $0 \leq u \leq t$ be an integer and $1 \leq h \leq d$. Then,
\[
\norm{D_h^u(M)} \lesssim C^2 W_{s-t}.
\]
\end{claim}

\begin{proof}
For $j,k \in \Zdst$, first observe that, since $s-d \geq t \geq 1$
\begin{align*}
\abs{\ip{f_k}{f_j}} \leq
C^2 \int_{\Rdst} (1+\abs{x-k})^{-s} (1+\abs{x-j})^{-s} dx 
\lesssim C^2 (1+\abs{k-j})^{-s}.
\end{align*}
Consequently,
\begin{align*}
\abs{D_h^u(M)_{k,j}} & \leq (1+\abs{k-j})^u \abs{\ip{f_k}{f_j}}
\\
&\lesssim C^2 (1+\abs{k-j})^t (1+\abs{k-j})^{-s}
= C^2 (1+\abs{k-j})^{-(s-t)}.
\end{align*}
Therefore
\begin{align*}
\sup_k \sum_j \abs{D_h^u(M)_{k,j}},
\quad
\sup_j \sum_k \abs{D_h^u(M)_{k,j}} \lesssim C^2 W_{s-t}.
\end{align*}
Consequently, by Schur's Lemma
\[
\norm{D_h^u(M)} \lesssim C^2 W_{s-t}.
\]
\end{proof}

\begin{claim}
\label{claim1.5}
Let $0 \leq u \leq t$ be an integer and $1 \leq h \leq d$. Then,
$D_h^u(M^{-1})$ is a bounded operator on $\ell^2(\Zdst)$.
\end{claim}
\begin{proof}
In Claim \ref{claim1} we have shown that the matrix $M$ satisfies the off-diagonal decay
condition,
\begin{align*}
\abs{M_{k,j}} \leq K (1+\abs{k-j})^{-s},
\end{align*}
for some constant $K>0$ that depends on $d$ and $C$.

Jaffard's Theorem (\cite{ja90}, see also \cite{gr04}) implies that $M^{-1}$ satisfies
a similar off-diagonal decay condition,
\begin{align*}
\abs{M^{-1}_{k,j}} \leq K' (1+\abs{k-j})^{-s},
\end{align*}
for some constant $K'>0$. A calculation similar to that in Claim \ref{claim1}, shows
that for $0 \leq u \leq t$, $D_h^u(M^{-1})$ is a bounded operator (and gives a norm
estimate depending on the unknown constant $K'$).
\end{proof}

\begin{claim}
\label{claim2}
Let $1 \leq h \leq d$. Then,
\[
\norm{D_h^t(M^{-1})} \lesssim_{t^2} A^{-(t+1)} C^{2t} W_{s-t}^t.
\]
\end{claim}
\begin{proof}
Let $u$, $1 \leq u \leq t$ be an integer. Consider the map $L \mapsto D_h(L)$, defined for those matrices $L$ such that $D_h(L)$ is a bounded operator. We know from Claims \ref{claim1} and \ref{claim1.5} that we can apply $D_h$ to both $M$ and $M^{-1}$, up to $t$ times. The map $D_h$ is a derivation from its domain into $B(\ell^2)$. This means that it satisfies the equation $D_h(AB) = D_h(A) B + A D_h(B)$, provided that $D_h(A)$ and $D_h(B)$ are both well defined. Iterating this equation yields,
\begin{align}
\label{tosolve}
0 = D_h^u(M^{-1}M) = \sum_{l=0}^u \binom{u}{l} D_h^l(M^{-1}) D_h^{u-l}(M).
\end{align}
Since all the operators involved in the last formula are bounded, we can associate factors to obtain,
\begin{align}
\label{solved}
D_h^u(M^{-1}) = -\sum_{l=0}^{u-1} \binom{u}{l} D_h^l(M^{-1}) D_h^{u-l}(M) M^{-1}.
\end{align}
It follows that
\begin{align*}
\norm{D_h^u(M^{-1})} \lesssim A^{-1} C^2 W_{s-t} \sum_{l=0}^{u-1} \binom{u}{l} \norm{D_h^l(M^{-1})}.
\end{align*}

Consider the numbers $v_u := \max_{0 \leq u' \leq u} \norm{D_h^{u'}(M^{-1})}$.
Using Claim \ref{claim1}, we observe that $A \leq \inf\sett{\abs{\lambda}: \lambda \in \sigma(M)} \leq \norm{M} \lesssim C^2 W_{s-t}$, where
$\sigma(M)$ denotes the spectrum of the Gramian matrix $M$ as an operator on $\ell^2$.

Consequently $1 \lesssim A^{-1} C^2 W_{s-t}$ and the numbers $v_u$ satisfy,
\begin{align*}
v_u &\lesssim A^{-1} C^2 W_{s-t} \sum_{l=0}^{u-1} \binom{u}{l} v_{u-1} 
\leq A^{-1} C^2 W_{s-t} 2^u v_{u-1},
\\
v_0 &\leq A^{-1}.
\end{align*}

Iterating the last relation yields,
\begin{align*}
v_t \lesssim_{t} A^{-(t+1)} C^{2t} W_{s-t}^t 2^\frac{t(t+1)}{2}
\approx_{t^2} A^{-(t+1)} C^{2t} W_{s-t}^t
\end{align*}
\end{proof}

Now we can prove the main result.
\begin{proof}[Proof of Theorem \ref{main}]
For $k,j \in \Zdst$, let $c_{kj} := \ip{g_k}{g_j}$. 
Since $M^{-1}_{kj} = c_{kj}$, using Claim \ref{claim2} we get,
\begin{align*}
\abs{k_h-j_h}^t \abs{c_{k,j}} \leq \norm{D_h^t(M^{-1})e_j}_2 \lesssim_{t^2}
A^{-(t+1)} C^{2t} W_{s-t}^t.
\end{align*}

Recalling that we are using the max-norm, we obtain
\begin{align*}
\abs{k-j}^t \abs{c_{k,j}} \lesssim_{t^2} A^{-(t+1)} C^{2t} W_{s-t}^t.
\end{align*}

For $k \not= j$, $\abs{k-j} \geq 1$, so $(1+\abs{k-j})^t \lesssim_t \abs{k-j}^t$. Consequently,
\begin{align*}
(1+\abs{k-j})^t \abs{c_{k,j}} \lesssim_{t^2} A^{-(t+1)} C^{2t} W_{s-t}^t.
\end{align*}
holds for $k \not= j$. For $k=j$ we use that $\abs{c_{kk}} = \norm{g_k}^2 \leq A^{-1}$.
Indeed, since $\sett{f_k}_k$ is a Riesz sequence with lower bound $A$, it is also a
frame sequence with the same lower bound. That is, the inequality,
$\sum_k \abs{\ip{f}{f_k}}^2 \geq A \norm{f}^2$ holds for all $f \in S$. Setting $f=g_j$, 
it follows that $\norm{g_j}^2 \leq A^{-1}$.

Since $W_s, C \geq 1$ and we have that
\begin{align*}
(1+\abs{k-j})^t \abs{c_{k,j}} \lesssim_{t^2} (A^{-(t+1)}+A^{-1}) C^{2t} W_{s-t}^t.
\end{align*}
for all $k,j \in \Zdst$.

Since for each $k \in \Zdst$, $g_k = \sum_j c_{k,j} f_j$, Lemma \ref{coef}
implies that
\[
\abs{g_k(x)} \lesssim_{t^2} (A^{-(t+1)}+A^{-1}) C^{2t+1} W_{s-t}^t (1+\abs{x-k})^{-t},
\quad x \in \Rdst.
\]
By Lemma \ref{obs}, $W_{s-t} \lesssim \left(1+\frac{1}{(s-t-d)}\right)$.
Therefore $W_{s-t}^t \lesssim_t \left(1+\frac{1}{(s-t-d)}\right)^t$.
\end{proof}
We conclude that,
\begin{align}
\label{almost}
\abs{g_k(x)} \lesssim_{t^2} (A^{-(t+1)}+A^{-1}) C^{2t+1}  \left(1+(s-t-d)^{-1}\right)^t (1+\abs{x-k})^{-t},
\quad x \in \Rdst.
\end{align}
This is almost the desired estimate; there is an extra $A^{-1}$ term. We get rid of it
using the following argument.

Let $\alpha$ be a positive real number. Consider the functions
$\widetilde{f_k} := \alpha f_k$. They form a Riesz sequence with lower and upper bounds
$\alpha^2 A$ and $\alpha^2 B$ respectively. Moreover, the functions forming the dual basis are given by $\widetilde{g_k} := \alpha^{-1} g_k$. The family $\sett{\widetilde{f_k}}_k$ satisfies a concentration condition similar to the one possessed by the original functions, but with a constant $\alpha C$. We apply the estimate in equation \eqref{almost} to this new family and learn that, for $x \in \Rdst$,
\begin{align*}
\abs{\widetilde{g_k}(x)} \lesssim_{t^2} 
\left(\alpha^{(-2t-2)} A^{-(t+1)}+\alpha^{-2} A^{-1}\right)
\alpha^{(2t+1)} C^{2t+1}
\left(1+(s-t-d)^{-1}\right)^t (1+\abs{x-k})^{-t}.
\end{align*}
Therefore, for $x \in \Rdst$,
\begin{align*}
\abs{g_k(x)} \lesssim_{t^2} 
(A^{-(t+1)}+\alpha^{2t} A^{-1}) C^{2t+1}
\left(1+(s-t-d)^{-1}\right)^t (1+\abs{x-k})^{-t}.
\end{align*}
Letting $\alpha \longrightarrow 0^+$, we get the desired estimate.
\section{Final remarks}

\begin{rem}
\label{rem3}
The most delicate part of the proof was the justification of the formal computations in Claim \ref{claim2}, that allowed us to solve $\norm{D^u_h(M^{-1})}$ recursively from the binomial formula. In order to associate factors, we needed to know that $M^{-1}$ belongs to the domain of $D^u_h$. 

To see why this is important, let us consider the case where the original basis is formed by the integer shifts of a single function. In this case, the Gramian matrix $M$ is a convolution operator, having some sequence $a$ as kernel. The matrix $M^{-1}$ is also a convolution operator and has a kernel $b$ that satisfies,
\begin{equation}
\label{aconvb}
a*b=\delta.
\end{equation}
The decay of $a$ and $b$ can be reformulated in terms of smoothness estimates for their Fourier transforms $\hat{a}$ and $\hat{b}$. The argument in the proof of Theorem \ref{main} amounts, in this case, to transferring smoothness estimates from $\hat{a}$ to its pointwise inverse $\hat{b}$ by an iterated application of the Leibniz product rule (cf. Equation \eqref{tosolve}.)

The obstacle to derive Equation \ref{solved} formally from Equation \ref{tosolve} is that the latter equation does not determine, by itself, the derivatives of $M^{-1}$.
For example, when $a$ is a finitely supported sequence, Equation \eqref{aconvb} is a recurrence equation in $b$, that has many solutions even if $\hat{a}$ has no zeros. The sequence $b$ that we are looking for (that is, the kernel of $M^{-1}$) can be singled out as the only solution of Equation \eqref{aconvb} that belongs to $\ell^2$.

In the case of a basis formed by the integer shifts of a single function, the justification
we need follows from some careful regularization argument for Sobolev spaces. In our case, this justification was done in Claim \ref{claim1.5}, by resorting to Jaffard's result \cite{ja90}. Another possible approach would be to use the general theory of unbounded derivations, in particular the results in \cite{brro75} and \cite{kish93}. However, this would require adapting those results to non-densely defined derivations.

The use of derivations is somehow implicit in Jaffard's original proof \cite{ja90}.
A recent article from Gr\"ochenig and Klotz \cite{grkl09} exploits in great depth the use of derivations to study off-diagonal decay of infinite matrices and approximation by classes of banded matrices.
\end{rem}

\begin{rem}
\label{rates}
As observed in the introduction, in Theorem \ref{main}, the decay condition on the original basis $\sett{f_k}_k$ is not shown to be fully shared by the dual basis $\sett{g_k}_k$ (although the results in \cite{gr04} show that the full decay condition is actually preserved.) This is due to the kind of objects used to bound the decay of the entries of $M$ and $M^{-1}$. According to Remark \ref{rem3} above, in the case of a basis formed by the integer translates of a single function, the estimates given amount to smoothness estimates for the symbol $\tau$ of some convolution operator. In Claims \ref{claim1} and \ref{claim2} we bounded the size of the entries of a matrix by means of its $\ell^2 \to \ell^2$ operator norm and controlled that norm by interpolating its $\ell^1 \to \ell^1$ and $\ell^\infty \to \ell^\infty$ norms (by Schur's lemma.) This would correspond in the case of a convolution operator to bounding the $L^\infty$ norm of its symbol $\tau$, from above by its $L^1$ norm and from below by its $L^2$ norm. This accounts, in that case, for the loss of some precision in the estimates.
\end{rem}
\begin{rem}
Finally, it should be noticed that the techniques used in this article could be applied to a general \emph{polynomially self-localized basis} in the abstract sense of \cite{fogr05}, where a basis is called self-localized if its Gramian matrix presents certain off-diagonal decay.
\end{rem}

\section*{Acknowledgements}
The author thanks Hans Feichtinger and Ursula Molter for some insightful discussions and Karlheinz Gr\"ochenig for his comments and for sharing an early draft of \cite{grkl09}.

Also, the author holds a fellowship from the CONICET and thanks this institution for its support. His research is also partially supported by grants: PICT06-00177, CONICET PIP 112-200801-00398 and UBACyT X149.

This note was partially written during a long-term visit to NuHAG in which the author was
supported by the EUCETIFA Marie Curie Excellence Grant (FP6-517154, 2005-2009).
\bibliographystyle{plain}

\end{document}